\pdfoutput=1
\documentclass[reqno]{amsart}
\usepackage{amssymb}
\usepackage{amsmath}
\usepackage{amsthm}
\usepackage{pgf}
\usepackage{color}
\usepackage{graphicx}   
\usepackage{hyperref}
\newtheorem{theorem}{Theorem}[section]
\newtheorem{lemma}[theorem]{Lemma}

\theoremstyle{definition}

\numberwithin{equation}{section}
\renewcommand{\d}{\mathrm{d}} 
\newcommand{\dt}{\d t}
\newcommand{\epsi}{\varepsilon}
\newcommand{\Rz}{{\mathbb R}}

\newcommand{\disp}{\displaystyle}

\newcommand{\calF}{\mathcal F}
 \newcommand{\calG}{\mathcal G}
 \newcommand{\calE}{\mathcal E}
\newcommand{\e}{{\rm e}}
\newcommand{\lan}{\langle}
\newcommand{\ran}{\rangle}
\begin{document}

\title[Penalization via global  functionals]{Penalization via global
  functionals of \\ optimal-control problems for dissipative evolution}

\author{Lorenzo Portinale}
\address[Lorenzo Portinale]{IST Austria,
Am Campus 1,
3400 Klosterneuburg,
Austria.}
\email{lorenzo.portinale@ist.at}

\author{Ulisse Stefanelli}
\address[Ulisse Stefanelli]{Faculty of Mathematics, University of Vienna, 
Oskar-Morgenstern-Platz 1, 1090 Wien, Austria  and  Istituto di Matematica
Applicata e Tecnologie Informatiche \textit{{E. Magenes}}, v. Ferrata 1, 27100
Pavia, Italy.}
\email{ulisse.stefanelli@univie.ac.at}
\urladdr{http://www.mat.univie.ac.at/$\sim$stefanelli}


\begin{abstract}
We consider an optimal control problem for an abstract nonlinear
dissipative 
evolution equation. The differential constraint is penalized by
augmenting the target functional by a nonnegative global-in-time
functional which is null-minimized iff the evolution equation is
satisfied. Different variational settings are presented, leading to
the convergence of
the penalization method for gradient flows, noncyclic and semimonotone flows, doubly
nonlinear evolutions, and GENERIC systems. 
\end{abstract}

\maketitle

%%%%%%%%%%%%%%%%%%%%%%%%%%%%%%%%%%%%%%%%%%%%%%
\section{Introduction}

We are concerned with the abstract optimal control problem
\begin{equation}
  \label{eq:1}
  \min\{F(u,y) \ : \ y \in S(u)\}.
\end{equation}
Here, $u:[0,T]\to H $ stands for a time-dependent admissible control, $H$ is a
Hilbert space, and $y:[0,T]\to H$  belongs to the set $S(u)$ of
 a nonlinear evolution equation  with datum $u$  to be
specified below. The nonnegative {\it target}
functional $F$ is defined on the trajectories $u$ and $y$.

Relation
$y  \in   S(u)$  corresponds  to different models of
dissipative evolution. In particular, we will  consider
the case of $u$-forced
\begin{align*}
  &\text{Gradient flows:}&&y'+\partial \phi(y) = u,\\
&\text{Monotone and pseudomonotone flows:}&&y'+A(y) = u,\\
&\text{Generalized gradient flows:}&&\partial_{y'} \psi(y,y')+\partial
  \phi(y) = u,\\
&\text{GENERIC flows:}&&y' =L(y) \,DE(y) - K(y) (\partial \phi(y) -u).
\end{align*}
The reader is referred to the following sections for all necessary details.
 In all of these cases, the abstract relation
$y \in  S(u)$  stands for the variational formulation of  a nonlinear partial differential
problem of parabolic type, possibly being singular or degenerate. 
% As introductory example let $y=S(u)$ correspond to the gradient flow $y' + \partial
% \phi(y)\ni u$ in
%  $H=L^2(\Omega)$ for some open,
% bounded, Lipschitz set $\Omega \subset \Rz^d$ and for the convex
% function $\phi: H \to
% (-\infty,\infty]$ given by $\phi(u)= \int_\Omega  \haz \beta(\nabla u)\d
% x$ where $\haz \beta:\Rz^d \to \Rz$ is convex, $\haz \beta(\xi) \geq
% c|\xi|^p$ for some $p>1$ and $c>0$, and $\phi$ is defined as $\infty$
% out of $W^{1,p}(\Omega)$. Then, relation $y
% \in S(u)$ actually corresponds to the classical weak formulation  in the
% quasilinear flow $y_t - {\rm div} (D\haz \beta (\nabla y)) \ni u $
% complemented by homogeneous Neumann boundary conditions. In
% particular, the homogeneous choice $\haz \beta(\xi) = |\xi|^p/p$
% corresponds to the $p$-Laplacian equation $y_t -{\rm div} ( | \nabla
% y|^{p-2}\nabla y) \ni u $.

The differential constraint $y\in S(u)$ will be equivalently reformulated as
$$y  \in   S(u) \ \ \Leftrightarrow \ \  G(u,y)=0,$$
where  the  {\it constraining functional} $G$ is a nonnegative functional on entire trajectories.
This characterization is not new. In the specific case of a gradient flow
$y'+\partial \phi(y) = u$, where $\partial \phi$ stands for the
subdifferential of the convex energy $\phi:H \to (-\infty,\infty]$, two
possible choices of the constraint functional $G$ are given by the
{\it Brezis-Ekeland-Nayroles} functional
$$G_{\rm BEN}(u,y) = \int_0^T \Big(\phi(y)+\phi^*(u-y') -(u,y) \Big) \dt + \frac12\|
y(T)\|^2 - \frac12\| y_0\|^2 $$
and the {\it De Giorgi} functional
$$G_{\rm DG}(u,y) = \int_0^T \left(\frac12 \| y'\|^2  +\frac12 \|\partial
\phi(y)-u\|^2 - (u,y') \right)\dt +
\phi(y(T)) - \phi(y_0).$$
Here, $(\cdot,\cdot)$ and $\| \cdot \|$ denote the scalar product and the norm
in $H$, respectively. 
The trajectory $y$ is forced to assume the initial value $y(0)=y_0$ by
defining $G(u,y)=\infty$ otherwise. % Versions of the previous
% functional adapted to doubly nonlinear \cite{be} and monotone flows
% \cite{visintin08} are also available.

The focus of this note is on the penalization of problem \eqref{eq:1}
by
\begin{equation}
  \label{eq:2}
  \min E_\epsi(u,y) \quad \text{for} \ \ E_\epsi(u,y):= F(u,y) + \frac{1}{\epsi}G(u,y).
\end{equation}
This  corresponds to approximate  the constrained
minimization of problem \eqref{eq:1} by means of  a family of  unconstrained
minimizations. 

This approach is indeed classical and has to be traced
back to {\sc Lions}  \cite{Lions68},  who proposed to penalize the
constraint by the
residual of the equation.  This has already been investigated,   
both in the stationary and the evolutive case, see
\cite{Bergounioux92,Bergounioux94,Bergounioux98,Gariboldi09,Gunzburger00,Mophou11}
among many others.  We follow this line by penalizing the
minimization by the De Giorgi
functional $G_{\rm DG}$, which corresponds to the residual by
nonetheless exploiting the variational structure
of the equation in order to simplify the energy. On the
other hand,  penalization in coordination with
 the Brezis-Ekeland-Nayroles functional $G_{\rm BEN}$ is not
directly related with residual minimization and, to our knowledge,  has not been studied
yet. Note that
the actual choice
of the constraining functional $G$  strongly
influences the properties of the
problem, so that the considering different options for $G$ is a
sensible issue.

In the case of the Brezis-Ekeland-Nayroles functional
$G_{\rm BEN}$, problem \eqref{eq:2} turns out to be a {\it separately convex}
minimization problem. This allows for the implementation of an alternate
minimization procedure, where $E_\epsi$ is alternatively minimized in
the state and the control until convergence.

The case of the De Giorgi functional $G_{\rm DG} $ bears its interest
in the fact that it is not restricted to convex functionals $\phi$. In
fact, $G_{\rm DG} $ is suited for nonconvex potentials as well and it can be easily modified to accommodate additional nonlinear features, such as nonlinear
dissipative or conservative terms (see Section \ref{sec:DG} below). 

Our aim is that of checking the solvability of the penalized
minimization problem
\eqref{eq:2} and the convergence of its minimizers to minimizers of
the constrained problem \eqref{eq:1} as $\epsi \to
0$. This will be achieved by proving the $\Gamma$-convergence of the
penalized functional
$E_\epsi $ to the limit $E_0$ defined by 
$$E_0(u,y)= F(u,y) \ \ \text{if} \ \ G(u,y)=0 \ \ \text{and} \ \
E_0(u,y)=\infty \ \ \text{otherwise}$$
under different variational settings, corresponding to the
above-mentioned different evolution  models. 

The paper is organized as follows.
The abstract functional setup is detailed in Section
\ref{sec:abstract}. Then, the application of the abstract theory to
the case of the Brezis-Ekeland-Nayroles variational principle for
gradient, noncyclic and semimonotone flows, and doubly nonlinear flows is
addressed in Section \ref{sec:BEN}. Eventually, Section \ref{sec:DG}
deals with the applications of De Giorgi principle in the context of gradient,
doubly nonlinear, and GENERIC flows.

 %%%%%%%%%%%%%%%%%%%%%%%%%%%%%%%%%%%%%%%%%%%%%%
\section{Abstract setup}\label{sec:abstract}

Let us start by specifying some notation. In the following, $H$ stands
for a real separable Hilbert space with scalar product $(\cdot,\cdot)$
and norm $\| \cdot \|$. The norm in the general Banach space $E$ will
be denoted by $\| \cdot \|_E$. Given the reference time $T>0$, we make
use of the standard Bochner spaces $L^p(0,T;E)$, $W^{1,p}(0,T;E)$,
$C([0,T];E)$ and so on.

A caveat on notation: we will use the same symbol $c$ to indicate positive
universal constants, possibly depending on data, and changing from
line to line.

Given a topological space $(X,\tau)$, we recall that a sequence of
functionals $\calE_\epsi : (X,\tau)\to [0,\infty]$ is said to
$\Gamma$-converge \cite{DeGiorgi79}
to the limit $\calE_0 : (X,\tau)\to [0,\infty]$ if $\calE_0(x) \leq
\liminf_{\epsi \to 0} \calE_\epsi(x_\epsi)$ for any $x_\epsi \to x$
and for all $\hat x\in X$ there exists a sequence $\hat x_\epsi\to
\hat x$ such that
$\calE_\epsi(x_\epsi)\to \calE_0(\hat x)$. The reader is referred to
{\sc Dal Maso} \cite{DalMaso93} for  a thorough presentation.  

We record here the following elementary lemma, which serves as basis
for proving convergence of the minimizers of problem \eqref{eq:2}
throughout.

\begin{lemma}[$\Gamma$-convergence]\label{lem} Let $(X,\tau)$ be a
   sequential  topological
  space and the functionals $\calF,\, \calG: (X,\tau) \to [0,\infty]$ be lower
  semicontinuous. Assume $\calE_\epsi : = \calF +\epsi^{-1}
  \calG$  to be  proper ($\calE_\epsi \not \equiv \infty$) and
  equicoercive for $\epsi>0$ small enough, namely that there exists
  $\epsi_0>0$, $\lambda>0$, and a compact $K  \subset X $
  such that $ \{x\in X \ : \ \calE_\epsi(x) <\lambda \} \subset K$ for all $\epsi <\epsi_0$.
Then, 
  \begin{itemize}
  \item[1.] $\calE_\epsi \stackrel{\Gamma}{\to} \calE_0 $ where
    $\calE_0(x) := \calF(x) $ if $ \calG(x)=0$ and $\calE_0=\infty$
    otherwise;
\item[2.] $\min \calE_\epsi$ can be solved for all
  $\epsi<\epsi_0$. Any sequence $x_\epsi$  of quasiminimizers, namely 
  $\liminf_{\epsi \to 0} (\calE_\epsi(x_\epsi) {-} \inf\calE_\epsi)=0$,
  admits a subsequence converging
  to a minimizer of~$\calE_0$;
\item[3.]  If  $ \calE_0$ admits a
  unique minimizer  $x_0$, any sequence of quasiminimizers of
  $\calE_\epsi$ converges to~$x_0$.
  \end{itemize}
\end{lemma}

\begin{proof}
  Ad 1. Let $x_\epsi \to x$ and assume with no loss of generality that
  $\sup_{\epsi} \calE_\epsi(x_\epsi)\leq c <\infty$. In particular, $0\leq \calG(x)
  \leq \liminf_{\epsi \to 0}\epsi\calE_\epsi(x_\epsi) \leq \liminf_{\epsi \to
    0}\epsi c=0$. Then $\calE_0(x)=\calF(x)
  \leq \liminf_{\epsi\to 0}\calF(x_\epsi) \leq \liminf_{\epsi\to
    0}\calE_\epsi(x_\epsi) $.  Fix now any $ \hat x\in
  X$. As $\epsi^{-1}\calG( \hat x) \to\infty$ if $\calG( \hat x)>0$, one has that 
$\calE_\epsi( \hat x) \to \calE_0(\hat x)$. This proves the $\Gamma$-convergence
$\calE_\epsi \stackrel{\Gamma}{\to} \calE_0$.

Ad 2. The existence of a minimizer $x_\epsi$ of $\calE_\epsi$ for
$\epsi <\epsi_0$ follows
from the equicoercivity and the lower semicontinuity of the sum $\calF
+\epsi^{-1} \calG$. Any sequence $x_\epsi$ of quasiminimizers belongs to $K$ for
$\epsi$ small enough. As such, it admits a subsequence (not relabeled)
converging to $x_0$ and, for any
$x\in X$, we have that 
$\calE_0(x_0)\leq \liminf_{\epsi \to 0} \calE_\epsi(x_\epsi) =
\liminf_{\epsi \to 0} \min \calE_\epsi \leq \liminf_{\epsi \to
  0}\calE_\epsi(x)=\calE_0(x)$. In particular, $x_0$ minimizes
$\calE_0$.

Ad 3.  This follows from the uniqueness of the minimizer of
$\mathcal E_0$ and from the fact that the topology is assumed to be sequential. 
\end{proof}

 %%%%%%%%%%%%%%%%%%%%%%%%%%%%%%%%%%%%%%%%%%%%%%
\section{Brezis-Ekeland-Nayroles principle}\label{sec:BEN}

In this section, we investigate penalization \eqref{eq:2} by letting
the constraining functional  to  be of Brezis-Ekeland-Nayroles type. Let us start by
presenting a result in the case of the classical gradient flow 
with forcing $u$  
\begin{equation}
  \label{eq:gf0}
  y'+\partial \phi(y)\ni u \ \ \text{in $H$, a.e. in} \ (0,T), \ \ y(0)=y_0.
\end{equation}
As usual, the prime denotes here derivation with respect to time. The
potential $\phi:H \to (-\infty,\infty]$ is assumed to be convex,
proper, and lower semicontinuous, and we denote by $D(\phi)=\{y\in H \
: \ \phi(y)<\infty\}$ its essential domain. The symbol
$\partial \phi$ denotes the corresponding subdifferential in the sense
of convex analysis. This is  defined as
$$ \xi \in \partial \phi(y) \ \ \Leftrightarrow   \ \ y \in D(\phi) \
\ 
\text{and} \ \ ( \xi,x-y ) \leq \phi(x) - \phi(y) \ \ \forall x
\in H.$$
The initial datum $y_0$ is assumed to belong to $D(\phi)$. Given $u\in
L^2(0,T;H)$, the solution $y\in H^1(0,T;H)$ 
of \eqref{eq:gf0} exists
uniquely \cite{Brezis73}. The
celebrated result by Brezis \& Ekeland \cite{Brezis-Ekeland76,Brezis-Ekeland76b} and Nayroles \cite{Nayroles76,Nayroles76b}
implies that $y$ solves \eqref{eq:gf0} iff $G_{\rm BEN} (u,y) =0$,
where the constraining functional $G_{\rm BEN} (u,y) :
L^2(0,T;H)\times H^1(0,T;H)$ is given by
\begin{equation}G_{\rm BEN} (u,y) =
\left\{
  \begin{array}{ll}
   \disp \int_0^T \Big( \phi(y) {+} \phi^*(u{-}y'){-}( u,y)\Big) \, \dt + \frac12
    \| y(T)\|^2 - \frac12\|y_0 \|^2 & \ \text{if} \ \ y(0)=y_0\\
\infty&\quad \text{otherwise}.
  \end{array}
\right.\label{eq:BEN}
\end{equation}
 Here, $\phi^*$ denotes the conjugate to $\phi$, namely,
$\phi^*(y^*) = \sup_y( (y^*,y) - \phi(y)$. 
Note that, for all $(u,y) \in L^2(0,T;H)\times H^1(0,T;H)$ the
functions $t \mapsto \phi(y'(t))$ and $y \mapsto \phi^*(u(t){-}y'(t))$
are measurable, so that $G_{\rm BEN} (u,y)$ is well defined. Still,
$G_{\rm BEN} (u,y)$ takes the value $\infty$ if $t \mapsto
\phi(y'(t))$ or $y \mapsto \phi^*(u(t){-}y'(t))$ do not belong to
$L^1(0,T)$.

Existence results based in the Brezis-Ekeland-Nayroles principle have
been obtained by {\sc Rios} \cite{Rios76}, {\sc Auchmuty}
\cite{Auchmuty93},  {\sc Roub\'\i\v cek} \cite{Roubicek00}, and
{\sc Ghoussoub \& Tzou} \cite{Ghoussoub-Tzou04} among others. In \cite{Ghoussoub-Tzou04}, the authors recast the problem within the far-reaching
theory of (anti-)selfdual  Lagrangians \cite{Ghoussoub08}. A variety
of extensions have been proposed, including perturbations
\cite{Ghoussoub-McCann04}, long-time dynamics \cite{Lemaire96},
measure data \cite{Mabrouk00}, time discretizations \cite{be2}, second-order \cite{Mabrouk03},
doubly-nonlinear \cite{be}, monotone \cite{visintin08}, 
pseudomonotone equations and their structural compactness \cite{visintin18}, and
rate-independent flows \cite{plas}. Note however that deriving
existence via these extensions may call for more stringent assumptions
on the data of the problem.

In the following, we will assume that the set of admissible controls
$U$ is a compact subset of $L^2(0,T;H)$. Moreover, we ask the target functional 
$F : L^2(0,T;H)\times H^1(0,T;H)\to [0,\infty)$ to be lower semicontinuous
with respect to the strong $\times$ weak topology of $L^2(0,T;H)\times
 H^1(0,T;H)$. An example in this class is
\begin{equation*}
F(u,y) =\frac12 \int_0^T \| y - y_{\rm target}\|^2\dt +\frac12 \int_0^T \| y' - y_{\rm target}'\|^2\dt +
\frac12\int_0^T\| u \|^2\dt %\label{eq:F}
\end{equation*}
for some given $y_{\rm target}\in H^1(0,T;H)$.
The main result of this section is the following.

\begin{theorem}[Gradient flows, BEN principle]\label{thm:BEN}
Let $\phi:H\to (-\infty,\infty]$ be convex, proper, and lower
semicontinuous, $y_0\in D(\phi)$, $\emptyset \not = U\subset\subset L^2(0,T;H)$, $F: L^2(0,T;H)
\times H^1(0,T;H)\to  [0,\infty]$  lower semicontinuous and coercive
w.r.t. the
strong $\times$ weak topology $\tau$ of $L^2(0,T;H)\times
H^1(0,T;H)$,  $F(u,y) <\infty$  only if  $u\in U $, $G_{\rm
  BEN}$ defined as in \eqref{eq:BEN}, and $E_\epsi:=
  F+\epsi^{-1}G_{\rm BEN}$ for $\epsi>0$. 

Then, $\min
  E_\epsi$ admits a solution for all
  $\epsi>0$. Moreover, $E_\epsi \stackrel{\Gamma}{\to} E_0$ with
  respect to topology $\tau$ where $E_0 =F$ on
    $\{G_{\rm BEN}=0\}$ and $E_0=\infty$ otherwise, and any sequence of
    quasiminimizers converges, up to a subsequence, to a solution of
    $\min E_0$. In case $\min E_0$ admits a unique minimizer, any
    sequence of quasiminimizers $\tau$-converges to it.
\end{theorem}

\begin{proof}
  In order to prove the statement we apply Lemma
  \ref{lem} with the choices $X= L^2(0,T;H)\times
  H^1(0,T;H)$ and $\tau =$ strong $\times$ weak topology in $X$. 

We start by checking that $E_\epsi$ is proper. In fact, by letting $u \in U$
and $y\in H^1(0,T;H)$ be the unique solution of $y'+\partial \phi(y)
\ni u$ with $y(0)=y_0$ we have that $E_\epsi(u,y) = F(u,y)<\infty$.

In order to prove the lower semicontinuity of $G_{\rm BEN}$, assume that $(u_n,y_n)\stackrel{\tau}{\to} (u,y)$. As 
$H^1(0,T;H)\subset
C([0,T];H)$ and $U$ is compact in $L^2(0,T;H)$ we have that  
\begin{align*}
&u_n - y_n' \to u- y' \ \ \text{weakly in} \ L^2(0,T;H), \\ 
&( u_n ,y_n ) \to ( u,y) \ \ \text{in} \ \ L^1(0,T), \\
&y_n(T)\to y(T) \ \ \text{weakly in} \ H.
\end{align*}
This implies that $G_{\rm BEN}(u,y) \leq \liminf_{n \to \infty}G_{\rm
  BEN}(u_n,y_n)$.
The equicoercivity of $E_\epsi$ follows from that of
$F$. 
\end{proof}

A remarkable feature of the penalization of  problem
\eqref{eq:1} via the Brezis-Ekeland-Nayroles functional relies in the
possibility of exploiting convexity. Indeed, in case $F$ is convex,
the penalized $F + \epsi^{-1}G_{\rm BEN}$ turns out to be separately
convex, the only nonconvexity coming from the bilinear term
$(u,y)$. This in turn suggests the possibility of implementing some
alternate minimization procedure. Note that, in
relation with applications to PDEs, the bilinear term $(u,y)$  is
usually  of lower order. %  If both $F$ and 
% $(y,w)\mapsto\phi(y)+\phi^*(w) - (y,w)$ are convex, the penalized
% functional $F + \epsi^{-1}G_{\rm BEN}$ turns out to be jointly
% convex. In this case, the optimal control problem
% \ref{eq:1}  is approximated by a sequence of convex minimization
% problems.

In the statement of Theorem \ref{thm:BEN} we have assumed $F$ to be
coercive. In fact, the functional $G_{\rm BEN}$ 
itself cannot be
expected to be coercive with respect to topology $\tau$. In
particular, this would follow by asking $\phi^*$ to be
superquadratic.  This  would however induce a quadratic bound to
$\phi$, a 
quite restrictive assumption, especially in relation
to PDEs. 

An alternative possibility is that of augmenting $G_{\rm BEN}$ by a coercive
term, which would still vanish on solutions of \eqref{eq:gf0}. A proposal in this
direction is in \cite{be}, where the following variant of the
Brezis-Ekeland-Nayroles functional is presented
\begin{equation}\tilde G_{\rm BEN} (u,y) = G_{\rm BEN} (u,y) + \left(
\int_0^T\big( \|y'\|^2{-}(u,y')\big) \, \dt + \phi(y(t)) - \phi(y_0)
\right)^+\label{eq:BEN2}
\end{equation}
with $r^+:=\max \{r,0\}$. By letting now
$E_\epsi = F + \epsi^{-1}\tilde G_{\rm BEN}$ one can prove the statement of
Theorem \ref{thm:BEN} also for a noncoercive functional $F$, for
coercivity for $y$ with respect to the weak topology of $H^1(0, T ;H)$ is provided by $\tilde G_{\rm BEN}$.

Before closing this subsection, let us remark that a time-dependent
potential $\phi$ can be considered as well, namely
\begin{equation}
  \label{eq:gf0t}
  y'(t)+\partial \phi(t,y(t))\ni u(t) \ \ \text{in $H$, for a.e.} \ t \in (0,T), \ \ y(0)=y_0.
\end{equation}
Here, $\phi:(0,T)\times H \to
(-\infty,\infty]$ is asked to be measurable with respect to ${\mathcal L}\otimes
{\mathcal B}(H)$, where ${\mathcal L} $ is the Lebesgue $\sigma$-algebra in $(0,T)$ and
${\mathcal B}(H)$ is the Borel $\sigma$-algebra in $H$, and such that $y \mapsto \phi(t,y)$ is proper,
convex, and lower
semicontinuous for a.e. $t \in (0,T)$. % Moreover, we assume that there exists $\pi: H^1(0,T;H)\to L^1(0,T) $ such that,
% for $ y\in H^1(0,T;H)$ and $\xi \in H^1(0,T;H)$  with  $\xi(t) \in \partial \phi(t,y(t))$ for a.e. $ t \in (0,T)$,  the mapping  $ t \mapsto \phi(t,y(t))\,$ is absolutely continuous and  satisfies
% \begin{equation} 
%   (\phi(t,y(t)))' = (\xi(t),y'(t)) + \pi(y)(t)\quad \text{for a.e. $ t \in (0,T)$}.\label{chain2}
% \end{equation}
Problem \eqref{eq:gf0t} can be equivalently reformulated as $G_{\rm
  BEN}(u,y)=0$ where 
$$
G_{\rm BEN} (u,y) =
\left\{
  \begin{array}{ll}
   \disp \int_0^T \Big( \phi(t,y(t)) {+} \phi^*(t,u(t){-}y(t)'){-}(
    u(t),y(t))\Big) \, \dt \\
\qquad{} +\disp\frac12
    \| y(T)\|^2 - \frac12\|y_0 \|^2 & \quad \text{if} \ \ y(0)=y_0\\
\infty&\quad \text{otherwise}.
  \end{array}
\right.
$$
where of course conjugation in $\phi^*$ is taken with respect to the second
variable only. In order to be sure, however, that pairs $(u,y)$ 
exist with  
that $G_{\rm BEN}(u,y) =0$,  some additional assumptions on the time
dependence $t \mapsto \phi(t,y)$ is required. The reader is referred
to \cite{Kenmochi77,Kenmochi91,Moreau77,Yamada76} for a collection of classical results in this direction.

\subsection{An example}%\label{sec:ex}  
With the aim of illustrating the statement of 
Theorem \ref{thm:BEN}, we investigate the ODE optimal control
problem 
\begin{align}
&\min\Bigg\{\frac12 \int_0^1 (y(t) - \e^{-t})^2\dt +\frac12 \int_0^1t^2 (u(t) -
  \e^{-t})^2\dt \ :\\
&\quad \quad  \quad    y'(t)+y(t)=u(t)\equiv u_0 \e^{-t},  \ u_0\in
  [0,1], \ 
  y(0)=1\Bigg\}.
\label{eq:case}
\end{align}
 Here, by taking advantage of the linearity of the
constraint one can directly compute $y(t)=S(u_0 \e^{-t})(t) = \e^{-t}(1+tu_0)$
and 
\begin{align*}
&u_0\mapsto F(S(u_0 \e^{-t}), u_0 \e^{-t}) :=\frac12 \int_0^1 (S(u_0 \e^{-t})(t)- \e^{-t})^2\dt +\frac12 \int_0^1t^2 (u_0 \e^{-t} -
  \e^{-t})^2\dt \\
&\quad= \left(\frac12 \int_0^1t^2 \e^{-2t}\dt\right)\left( u_0^2 +
  (u_0-1)^2 \right)=:\gamma \left( u_0^2 +
  (u_0-1)^2 \right)
\end{align*}
In particular, the optimal control corresponds to $u_0=1/2$, the
optimal solution is $y(t) = \e^{-t}(1+t/2)$, and the minimum of $E_0$
is 
$$F(\e^{-t}(1+t/2),\e^{-t}/2)=\gamma/2 = 1/16 -5/(16\e^2)\sim 0.0202.$$

The ODE is the gradient flow of the potential
$\phi(y)=y^2/2$ under the additional forcing $u$. Correspondingly,  the Brezis-Ekeland-Nayroles
functional $G_{\rm BEN}$ is given by 
$$G_{\rm BEN}(u,y) = 
\left\{
  \begin{array}{ll}
\disp\int_0^1\left(\frac12y^2 + \frac12 (u-y')^2 -uy
\right)\dt +\frac12 y^2(1) - \frac12& \ \ \text{if} \ y(0)=1,\\
\infty&\ \ \text{otherwise}.
  \end{array}
\right.
$$
The penalized optimal control problem reads then
\begin{align*}&\min\Bigg\{\int_0^1 \left( \frac12 (y(t) - \e^{-t})^2 +\frac{t^2}{2} (u(t) -
  \e^{-t})^2 + \frac{1}{2\epsi}y^2(t) + \frac{1}{2\epsi} (u(t)-y'(t))^2
  -\frac{1}{\epsi}u(t)y(t)\right)\dt \\
&\quad \quad \quad \quad +\frac{1}{2\epsi}y^2(1) -
\frac{1}{2\epsi}\  \ : \   \ u(t)\equiv u_0 \e^{-t},  \ u_0 \in [0,1],
  \ y(0)=1\Bigg\}.
\end{align*}
For all given $u$, the Euler-Lagrange equation for $E_\epsi = F +
\epsi^{-1}G_{\rm BEN}$ in terms of $y_\epsi$ is
 \begin{align*}
   &y''(t)-y(t)-\epsi y(t) = -(2u_0+\epsi)\e^{-t}, \quad y'(1) + y(1) =
     u_0/\e.
 \end{align*}
Complemented with the initial condition $ y(0)=1$, 
these linear relations
uniquely identify a critical point $y_\epsi$ of $E_\epsi$.  In
fact, this  is 
necessarily the  unique  minimizer of the convex functional $y \mapsto E_\epsi(u,y)$ and can
be explicitly
determined in terms of $u_0$ as
$$y_{\epsi,u_0}(t) = c_{1\epsi}\e^{-\alpha_\epsi t} +
c_{2\epsi}\e^{\alpha_\epsi t} + \left(\frac{2u_0}{\epsi} + 1 \right) \e^{-t}$$
where we have used the shorthand notation
\begin{align*}
  \alpha_\epsi&:=(1+\epsi)^{1/2},\\
c_{1\epsi}&:=\left(\frac{u_0}{\e} - (1+\alpha_\epsi)\left(
            \frac{2u_0}{\epsi}+1\right)
            \right)\left((1-\alpha_\epsi)\e^{-\alpha_\epsi} -  (1+\alpha_\epsi)\e^{\alpha_\epsi} \right)^{-1},\\
c_{2\epsi} &:= - \frac{2u_0}{\epsi} - c_{1\epsi}.
\end{align*} 
The value of $E_\epsi(u_0\e^{-t} ,y_{\epsi,u_0})$  
can be explicitly
evaluated. An elementary but tedious computation gives 
\begin{align}
&E_\epsi(u_0\e^{-t} ,y_{\epsi,u_0}) 
 = \left(\frac{c_{1\epsi}^2}{2}  +\frac{c_{1\epsi}^2}{2\epsi} +\frac{\alpha_\epsi^2c_{1\epsi}^2}{2\epsi}\right)
  \frac{\e^{-2\alpha_\epsi}-1}{-2\alpha_\epsi} + \left(\frac{c_{2\epsi}^2}{2} +\frac{c_{2\epsi}^2}{2\epsi}+\frac{\alpha_\epsi^2c_{2\epsi}^2}{2\epsi}\right)
  \frac{\e^{2\alpha_\epsi}-1}{2\alpha_\epsi}
\nonumber\\
& \quad + \left( \frac{2u_0^2}{\epsi^2}+\frac{1}{2\epsi}
  \left(\frac{2u_0}{\epsi}+1 \right)^2 +
  \frac{1}{2\epsi}\left(\frac{2u_0}{\epsi}+1+u_0
  \right)^2-\frac{u_0}{\epsi}\left(\frac{2u_0}{\epsi}+1 \right)\right)
  \frac{\e^{-2}-1}{-2}
\nonumber\\
& \quad + \left(\frac{2c_{1\epsi}u_0}{\epsi}
  +\frac{c_{1\epsi}}{\epsi}\left(\frac{2u_0}{\epsi}+1 \right)
  +\frac{\alpha_\epsi c_{1\epsi}}{\epsi}\left(\frac{2u_0}{\epsi}+1+u_0
  \right) - \frac{c_{1\epsi}u_0}{\epsi}\right)
  \frac{\e^{-\alpha_\epsi-1}-1}{-\alpha_\epsi-1}
\nonumber\\
& \quad + \left(\frac{2c_{2\epsi}u_0}{\epsi}  +\frac{c_{2\epsi}}{\epsi}\left(\frac{2u_0}{\epsi}+1 \right) -\frac{\alpha_\epsi c_{2\epsi}}{\epsi}\left(\frac{2u_0}{\epsi}+1+u_0 \right) - \frac{c_{2\epsi}u_0}{\epsi}\right)
  \frac{\e^{\alpha_\epsi-1}-1}{\alpha_\epsi-1}
\nonumber\\
& \quad + \left(1+\frac1\epsi -\frac{\alpha_\epsi^2}{\epsi}\right) c_{1\epsi}c_{2\epsi}
 + \frac{1}{2\epsi}\left(c_{1\epsi} \e^{-\alpha_\epsi} +
  c_{2\epsi} \e^{\alpha_\epsi} + \left(\frac{2u_0}{\epsi}+1\right)\e^{-1}
  \right)^2 - \frac{1}{2\epsi} \nonumber\\
& \quad + \gamma (u_0-1)^2. \label{ill}
\end{align}
Different curves $u_0 \mapsto E_\epsi(u_0\e^{-t} ,y_{\epsi,u_0})$ for different
choices of $\epsi$ are depicted in Figure \ref{illustration}. We
observe that the minimizer  and the minimum approach $1/2$ and
$0.0202$, respectively,  as $\epsi\to 0$, as
expected.  
\begin{figure}[h]
  \centering
  \pgfdeclareimage[width=105mm]{BEN}{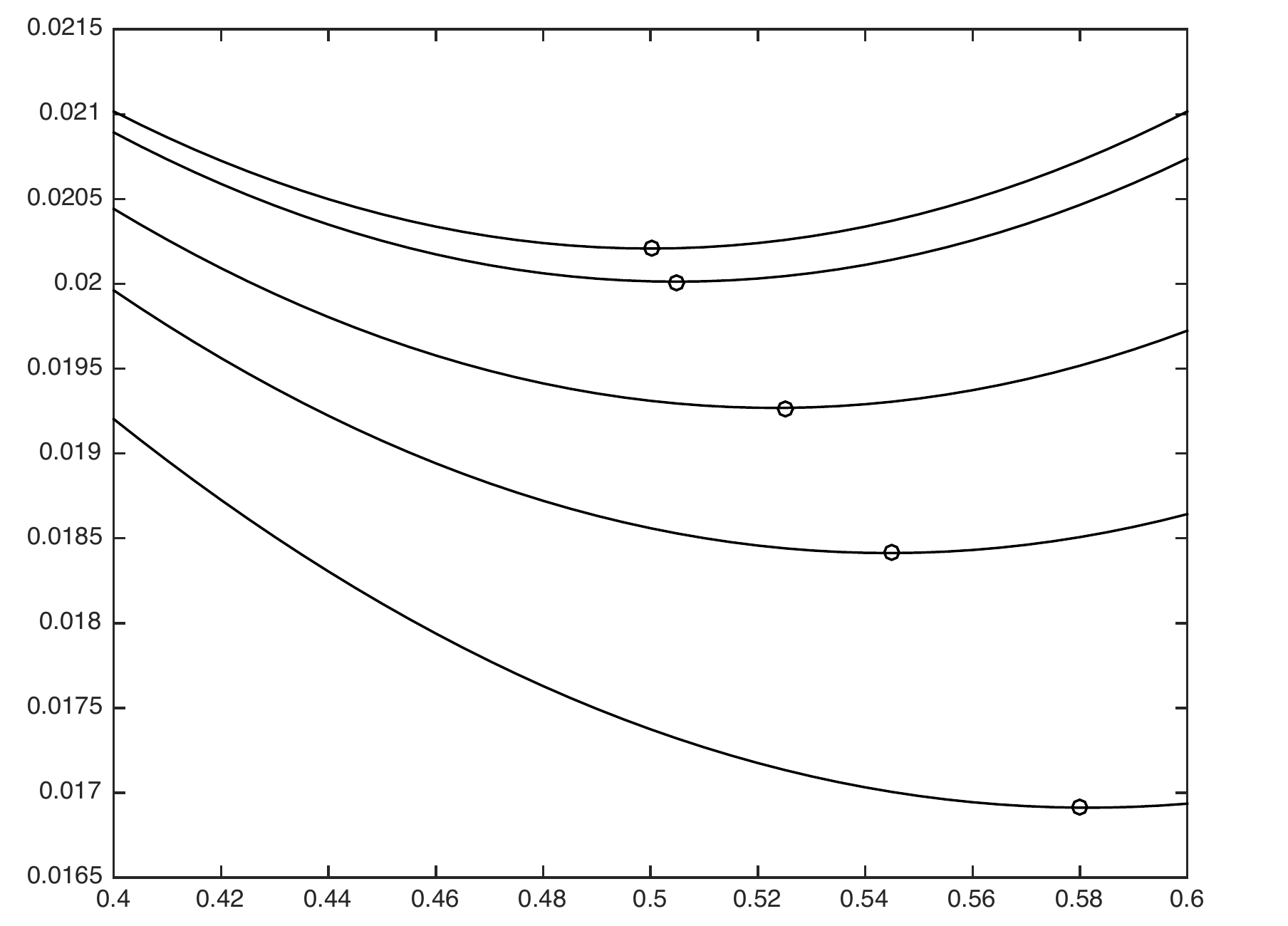}
    \pgfuseimage{BEN}
\caption{Curves $u_0 \mapsto E_\epsi(u_0\e^{-t} ,y_{\epsi,u_0})$ from
  \eqref{ill} for $\epsi=2,\, 1, \, 0.5, \, 0.1$, and $ 0 $ (bottom to top). On each curve, the dot indicates the
  minimizer.}
 \label{illustration}
\end{figure}

\subsection{Gradient flows in dual space} The statement of Theorem
\ref{thm:BEN} can be extended to the case of gradient-flow dynamics in
dual spaces. Let us 
introduce a real reflexive Banach space $W$, densely and continuously
embedded into $H$, so that $W\subset H \subset W^*$ is a classical
Gelfand triplet. We consider the problem
\begin{equation}
  \label{eq:gfff}
  y'+\partial \phi(y)\ni u \ \ \text{in $W^*$, a.e. in} \ (0,T), \ \ y(0)=y_0.
\end{equation}
The
potential $\phi:W \to \Rz$ is assumed to be  everywhere defined,  convex,
proper, and lower semicontinuous. The symbol
$\partial \phi$ in \eqref{eq:gfff} denotes now the subdifferential between $W$ and
$W^*$. This is defined as
$$ \xi \in \partial \phi(y) \ \ \Leftrightarrow   \ \
 \lan \xi,x-y \ran \leq \phi(x) - \phi(y) \ \ \forall x
\in W$$
where $\lan \cdot, \cdot \ran$ is the duality pairing between $W^*$
and $W$.  We assume $\phi$ to be bounded as follows
\begin{align*}
  &\phi(y)\geq c\|y\|_W^{ m} - \frac1c \quad \forall y \in W, \quad
  \phi^*(y^*)\geq c\|y^*\|_{W^*}^{ m'} - \frac{1}{c} \quad \forall y^* \in
  W^*\\
& \| \xi \|_{W^*}^{ m'} \leq c (1+ \| y \|_W^{ m}) \quad \forall y \in W, \, \xi
  \in \partial \phi(y)
\end{align*}  
 where $m>1$ and $m'=m/(m-1)$. 
In particular, the above bounds
entail a polynomial control on $\phi$ of the form 
$$\phi(y)\leq  c \| y\|_W^{ m}  + c \quad \forall y \in
W, \quad \phi^* (y^* )\leq  c \| y^* \|_{W^* }^{ m'}  + c  \quad \forall y^*  \in
W^*$$
 which is now compatible with PDE applications.

Given the initial datum $y_0\in W$ (recall that $D(\phi) =W$), for all $u\in
L^{ m'} (0,T;W^*)$, the solution 
$y\in  W^{1,m'}  (0,T;W^*)\cap L^{ m}(0,T,W)$
of \eqref{eq:gfff} exists
uniquely. In particular,  $y$ solves \eqref{eq:gfff} iff  $G_{\rm
  BEN} (u,y) =0$ where 
$$G_{\rm BEN} (u,y) =
\left\{
  \begin{array}{ll}
   \disp \int_0^T \big( \phi(y) + \phi^*(u-y')-\lan u,y\ran\Big) \, \dt + \frac12
    \| y(T)\|^2 - \frac12\|y_0 \|^2 & \quad \text{if} \ \ y(0)=y_0\\
\infty&\quad \text{otherwise}.
  \end{array}
\right.
$$
The result of Theorem \ref{thm:BEN} can be reformulated in this setting by assuming the set of admissible controls
$U$ to be a compact subset of $L^{ m'}(0,T;W^*)$ and 
$F : L^{ m'}(0,T;W^*)\times W^{1,m'}  (0,T;W^*)\cap L^{ m}(0,T,W)\to [0,\infty]$ to be lower semicontinuous
with respect to the strong $\times$ weak topology of $L^{
  m'}(0,T;W^*)\times W^{1,m'}  (0,T;W^*)\cap L^{
  m'}(0,T,W)$, with $F(u,y)<\infty$  only if  $u \in U$. Note that here no coercivity of $F$ is actually
  needed, for in this case $G_{\rm BEN} $ itself turns out to be coercive, due to
  the lower bounds on $\phi$ and $\phi^*$. 

Once again, $G_{\rm BEN}$ is proper, since it vanishes on solutions to
\eqref{eq:gfff}, which are known to exist.
In order to check for the
  lower semicontinuity of $G_{\rm
  BEN}$ one would need to recall the embedding $  W^{1,m'}  (0,T;W^*)\cap L^{ m}(0,T,W)\subset
C([0,T];H)$. In particular, the term $\| y(T)\|^2$ turns out to be lower semicontinuous.

\subsection{Nonpotential and nonmonotone flows}   Originally limited to gradient flows
of convex functionals, the Brezis-Ekeland-Nayroles variational approach
has been extended to classes of nonpotential monotone flows by  {\sc
  Visintin} \cite{visintin08}. By replacing Fenchel duality by the representation theory by {\sc
  Fitzpatrick} \cite{Fitzpatrick}, he noticed that solutions of the
nonpotential flow
\begin{equation}
  \label{eq:gf1}
y'+Ay\ni u \quad \text{a.e. in} \ W^*, \quad y(0)=y_0,
\end{equation}
where $A: W \to 2^{W^*}$ is a maximal monotone, coercive, and {\it representable} operator and $y_0\in
D(A)$, can be characterized by $G_{\rm BEN}(u,y)=0$, where $G_{\rm BEN}
: L^2(0,T;W^*)\times H^1(0,T;W^*)\cap L^{ 2}(0,T,W)\to [0,\infty]$ is now given as
\begin{equation}
G_{\rm BEN}(u,y)=
\left\{
  \begin{array}{ll}
   \disp \int_0^T\Big(f_A(y, u{-}y') {-} \lan u,y\ran\Big)\, \dt
    +\frac12\|y(T)\|^2 - \frac12 \| y_0\|^2& \ \ \text{if} \
                                             y(0)=y_0,\\
\infty &\ \ \text{otherwise.}
  \end{array}
\right.\label{eq:semi}
\end{equation}

The function $f_A:W\times W^*\to
(-\infty,\infty]$ is convex, lower semicontinuous, with $f_A(y,y^*)\geq
\lan y^*,y\ran$ for all $(y,y^*) \in W\times W^*$, and {\it
  represents} the operator $A$  in  the following sense
\begin{equation}
 y^* \in Ay \ \ \Leftrightarrow \ \ f_A (y,y^*)=\lan
y^*,y\ran.\label{eq:rep}
\end{equation}
An operator is said to be {\it representable} when it admits a representing
function. All maximal monotone operators are representable, for
instance via their {\it Fitzpatrick function}
$$f_A(y,y^*) := \lan y^*,y\ran + \sup\{\lan y^* - \tilde y^*, \tilde y-
y\ran \ : \ \tilde y \in W, \ \tilde y^* \in A\tilde y\}.$$
A monotone operator need however not be cyclic nor maximal to be
representable. The reader is referred to \cite{visintin17,visintin18}
for a full account on this theory. By taking advantage of position
\eqref{eq:semi}, the assertion of Theorem
\ref{thm:BEN} can hence be modified to include the case of the
differential constraint \eqref{eq:gf1} as well.

More generally, the reach of the penalization via the
Brezis-Ekeland-Nayroles functional extends even beyond monotone
situations. Assume to be given $B:H \times W \to 2^{W^*}$ such that 
\begin{align*}
  &B(h,\cdot): W  \to 2^{W^*} \ \ \text{is maximal monotone}, \
  \forall h \in H,\\[2mm]
& \forall (h,y)\in H\times W, \ \forall y^* \in B(h,y), \ \forall h_n
  \to h \ \text{in} \ H \\
&\text{there exists} \ y^*_n \ \text{such
  that} \ y^*_n \in B(h_n,y_n) \ \text{and} \ y^*_n \to y^* \
  \text{in} \ W^*.
\end{align*}
This class of nonmonotone operators $A(y) : = B(y,y)$, called {\it
  semimonotone}  \cite{visintin18}, includes the class of
{\it pseudomonotone} operators \cite{browder}, and it is representable
\cite[Thm. 4.4]{visintin18}  in the sense of
\eqref{eq:rep} by means of a
weakly lower semicontinuous albeit nonconvex function $f_A$
\begin{equation}
f_A(y,y^*) := \lan y^*,y\ran + \sup\{\lan y^* - \tilde y^*, \tilde y-
y\ran \ : \ \tilde y \in W, \ \tilde y^* \in B(y,\tilde
y)\}.\label{eq:fi}
\end{equation}
On this basis, the nonmonotone flow
\begin{equation}
  \label{eq:gf11}
y'+A(y)\ni u \quad \text{a.e. in} \ W^*, \quad y(0)=y_0,
\end{equation}
driven by the semimonotone operator $A(y)$ can be variationally
reformulated as $G_{\rm BEN}=0$, where $G_{\rm BEN}$ is defined in from \eqref{eq:semi}, where however $f_A$
is now defined by \eqref{eq:fi}. Note that $G_{\rm BEN}$ is proper and lower semicontinuous with respect to the
strong $\times$ weak topology of $L^2(0,T;W^*)\times H^1(0,T;W^*)\cap
L^{ 2}(0,T,W)$. By letting $E_\epsi = F+\epsi^{-1}G_{\rm BEN}$
and assuming again that $F$ is coercive and $F(u,y)<\infty$  only
if  $u\in U$, the results of Theorem
\ref{thm:BEN} can be extended to the case of optimal control problems
driven by \eqref{eq:gf11} as well.

\subsection{Doubly nonlinear flows}  A gradient flow  can be seen
as a particular case of the doubly nonlinear evolution
 \begin{equation}
  \label{eq:gf3}
  \partial\psi(y')+\partial \phi(y)\ni u \ \ \text{in $V^*$, a.e. in} \ (0,T), \ \ y(0)=y_0.
\end{equation}
Here, $V$ is a real reflexive Banach space with $W \subset \subset V$,
the symbol $\partial$  refers  to the subdifferential between $V$ and
$V^*$, and $\psi:V \to [0,\infty)$ is a second convex, proper, lower semicontinuous functional defined on the
whole $V$. More precisely, we assume $\psi$ to fulfill $0\in \partial
\psi(0)$  and to be of polynomial
growth, namely 
\begin{align*}
& c\|y'\|^p_V -\frac1c \leq \lan w , y' \ran, \quad \|
w\|_{V^*}^{p'} \leq c(1 + \| y'\|^p_V) \quad \forall y' \in V, \,
w \in \partial \psi(y')\\
&   
\psi^*(w)\geq c \| w\|^{p'}_{V^*} -\frac1c \quad \forall  w
\in V^*
\end{align*}
for $p >1$ and $p'=p/(p-1)$. Additionally, we assume $D(\phi)=W$ and 
the coercivity 
$$ \phi(y) \geq c \| y \|^m_W - \frac1c \quad \forall y \in W$$
for some $m>1$. In \cite{be} a doubly nonlinear version of
the Brezis-Ekeland-Nayroles functional is addressed.   In particular,
one has that $(u,y,w) \in L^{ p'}(0,T;V^*) \times
W^{1,p'}(0,T;V^*)\cap L^m(0,T;W)\times L^{ p'}(0,T;V^*)$
solve 
$$ w\in \partial \psi(y'), \quad \partial \phi(y)\ni u -w\ \ \text{a.e. in} \ (0,T), \ \ y(0)=y_0 $$
iff $G_{\rm BEN}(u,y,w)=0$, where $G_{\rm BEN}: L^{ p'}(0,T;V^*)
\times W^{1,p'}(0,T;V^*)\cap L^m(0,T;W)\times L^{ p'}(0,T;V^*)\to [0,\infty]$ is now defined as 
$$
  G_{\rm BEN}(u,y,w)=
\left\{
  \begin{array}{ll}
 \left(\disp\int_0^T\Big (\psi(y') + \psi^*(w) - \lan
    u,y'\ran\Big)\, \dt + \phi(y(T)) - \phi(y_0) \right)^+& \\
 \quad+ \disp\int_0^T \Big (\phi(y) + \phi^*(u-w) - \lan u-w, y \ran
  \Big)\, \dt& \ \text{if}
                                                            \
               y(0)=y_0\\
\infty& \ \text{otherwise}.
  \end{array}
\right.
$$
Indeed, the two nonnegative integrals in the definition of $G_{\rm
  BEN}$ correspond to the two relations $ w\in \partial \psi(y')$ and
$ \partial \phi(y)\ni u -w$, respectively. At the price of introducing
the new variable $w$, one can penalize the differential constraint
\eqref{eq:gf3} by minimizing $(u,y,w)\mapsto E_\epsi(u,y,w) =
F(u,y,w)+\epsi^{-1}G_{\rm BEN}(u,y,w)$. Again, the results of Theorem
\ref{thm:BEN} can be extended to this situation. In particular, it can be proved
that $G_{\rm BEN}$ is proper and lower semicontinuous with respect
to the strong $\times$ weak $\times$ weak topology of $L^{ p'} (0,T;V^*) \times W^{1,p'}(0,T;V^*)\cap L^m(0,T;W)\times L^{ p'} (0,T;V^*)$. Moreover, it turns out to be
coercive as well, as soon as it is restricted to $u\in U$. In particular, no coercivity
has to be assumed on $F$ in this case. Indeed, $G_{\rm
  BEN}$ is here the doubly nonlinear version of the former
\eqref{eq:BEN2}, which was in fact introduced to ensure coercivity. 

 %%%%%%%%%%%%%%%%%%%%%%%%%%%%%%%%%%%%%%%%%%%%%%
\section{De Giorgi principle}\label{sec:DG}

Let us now turn out attention to the penalization \eqref{eq:2} by means of
a variational reformulation of dissipative evolution, following the
general approach to gradient flows from
\cite{DeGiorgi80}. 

Consider again the classical gradient flow in a Hilbert space \eqref{eq:gf0}
% \begin{equation}
%   \label{eq:5}
%   y'+\partial_\ell\phi(y)\ni u \ \ \text{a.e. in} \ (0,T), \ y(0)=y_0
% \end{equation}
where now the potential $\phi:H \to
(-\infty,\infty]$ is asked to be lower semicontinuous and proper, possibly being
nonconvex. To keep notation to a minimum, let us 
assume $\phi = \phi_1 + \phi_2$ with $\phi_1$ convex, proper, and lower
semicontinuous, and $\phi_2\in C^{1,1}$. Then, by letting $\partial
\phi$ denote the classical {\it Fr\'echet subdifferential}, namely
$$\xi \in \partial \phi(y) \ \ \Leftrightarrow \ \ y \in D(\phi) \
 \ \text{and} \ \ \liminf_{w \to y} \frac{\phi(w) - \phi(y) -
   (\xi,w-y)}{\| y-w\|}\geq 0,$$
(note that the Fr\'echet subdifferential coincides with the 
subdifferential of convex analysis on convex functions) we have that $\partial
\phi= \partial \phi_1 + D\phi_2$. We will additionally assume $\partial
\phi_1$ to be single-valued, whenever nonempty.  More general settings
are discussed in Subsection \ref{sec:more} below.
%  a $C^{1,1}$ perturbation of a convex function. The symbol $\partial_\ell
%  \phi$ refers to the {\it limiting 
% subdifferential} \cite{Rossi-Savare06}. After defining the Frech\'et subdifferential
% $$\xi \in \partial_F \phi(y) \ \ \Leftrightarrow \ \ y \in D(\phi) \
% \ \text{and} \ \ \liminf_{w \to y} \frac{\phi(w) - \phi(y) -
%   (\xi,w-y)}{\| y-w\|}\geq 0,$$
% the limiting subdifferential $\partial_\ell\phi$ is defined as the
% strong-weak closure of $\partial_F\phi$ in $H^2$, along sequences of
% bounded energy. More precisely
% \begin{align*}\xi\in\partial_\ell \phi(y)  \ \ \Leftrightarrow& \ \
%   \exists (y_n,\xi_n)\in H^2 \ \text{such that}  \  \sup\phi(y_n)<\infty, \ \xi_n
% \in \partial \phi_F \phi(y_n), \\
% & \ \ \text{and} \  (y_n,\xi_n)\to
% (y,x) \ \text{strongly$\times$ weakly in} \ H^2.
% \end{align*}
% In the following, we will slightly abuse notation by using the symbol
% $\partial_\ell\phi$ as a single-valued function. In case this is however
% multivalued, additional care should be needed in order to specify
% selections.

Solutions to \eqref{eq:gf0} correspond to $G_{\rm DG}(u,y)=0$, where
the  functional $G_{\rm DG}: L^2(0,T;H)\times H^1(0,T;H) \to
[0,\infty]$  is defined as
\begin{equation}
 G_{\rm DG}(u,y)=
\left\{
  \begin{array}{ll}
    &\disp\int_0^T \left(\disp\frac12 \| y'\|^2+ \frac12 \|\partial
\phi(y){-}u\|^2 - (u,y') \right)\dt +
\phi(y(T))- \phi(y_0) \\[2mm]
 {}& \qquad  \text{if} \ y \in D(\partial \phi) \
                        \text{a.e.} \ \text{and} \ y(0)=y_0\\[2mm]
&\infty  \quad  \text{otherwise}.
  \end{array}
\right.\label{eq:DG0}
\end{equation}
Due to its ties with the variational theory of steepest decent in
metric spaces from \cite{DeGiorgi80} we  call  $G_{\rm DG}$ {\it De
  Giorgi} functional.
In \eqref{eq:DG0} we used the notation $D(\partial \phi)$ to indicate the
essential domain of $\partial \phi$, namely $D(\partial \phi)=\{y \in
H \ : \ \partial \phi(y)\not =\emptyset\}$. Note that, by \cite[Lemma
3.4]{Rossi-Savare06}, the map $t\mapsto \partial\phi(y(t))$ is
measurable whenever $y\in H^1(0,T;H)$ with $y \in D(\partial \phi)$
a.e. The reformulation of the gradient flow \eqref{eq:gf0} via $G_{\rm
DG}$ is based on the computation of the squared residual
of \eqref{eq:gf0}, namely,
\begin{align*}
&  \int_0^T \frac12\| y' +\partial \phi(y) -u\|^2\dt = 
  \int_0^T\left(\frac12\|y'\|^2+\frac12\| \partial \phi(y) -u\|^2+ (y', \partial
  \phi(y) -u)\right)\dt \\
&\quad= G_{\rm DG}(u,y) \ \ \ \text{if} \ y(0)=y_0.
\end{align*}
The latter computation hinges on the chain rule $(\partial \phi(y),y') = (\phi
\circ y)'$, which holds in the case of $\phi=\phi_1+ \phi_2$ in the
following precise form \cite[Lemme 3.3]{Brezis73}
\begin{align}
  &y \in H^1(0,T;H), \ \ \xi \in L^2(0,T;H), \ \ \xi \in \partial \phi(y) \
  \ \text{a.e. in} \ (0,T)\nonumber\\
&  \Rightarrow \ \ \phi \circ y \in
  AC(0,T) \ \ \text{and} \ \  (\phi \circ y)' = (\xi,y')  \  \ \text{a.e. in} \ (0,T). \label{eq:chain}
\end{align}
Indeed, note that $\partial \phi(y) \in L^2(0,T;H)$ if $G_{\rm BEN}(u,y)<\infty$.
The main result of this
section is the following.

\begin{theorem}[Gradient flows, DG principle]\label{thm:DG}
Let $\phi=\phi_1+\phi_2:H\to (-\infty,\infty]$ have compact sublevels
and fulfill the chain rule \eqref{eq:chain},
with $\phi_1 $ proper, convex, and lower
semicontinuous, $\partial \phi_1$ single-valued, and $\phi_2\in
C^{1,1}$. Moreover, let 
$y_0\in D(\phi)$, $\emptyset\not =U \subset\subset L^2(0,T;H)$, $F: L^2(0,T;H)
\times H^1(0,T;H)\to [0,\infty]$ be lower semicontinuous w.r.t. the
strong $\times$ weak topology $\tau$ of $L^2(0,T;H)\times
H^1(0,T;H)$, $F(u,y)<\infty$  only if  $u\in U$, $G_{\rm
  DG}$ be defined as in \eqref{eq:DG}, and $E_\epsi:=
  F+\epsi^{-1}G_{\rm DG}$ for $\epsi>0$. 

Then, $\min
  E_\epsi$ admits a solution for all $\epsi>0$. Moreover $E_\epsi
\stackrel{\Gamma}{\to} E_0$ with respect to topology $\tau$ where $E_0 =F$ on
    $\{G_{\rm DG}=0\}$ and $E_0=\infty$ otherwise, and any sequence of
    quasiminimizers converges, up to a subsequence, to a solution of
    $\min E_0$. In case $\min E_0$ admits a unique minimizer, any sequence of quasiminimizers converge to it with
    respect to $\tau$.
\end{theorem}

\begin{proof} The statement follows by applying Lemma \ref{lem} in
  the space $X=L^2(0,T;H)\times H^1(0,T;H)$ endowed with its strong $\times$ weak
  topology $\tau$.

Let $u \in U$
and let $y\in H^1(0,T;H)$ be the unique solution of $y'+\partial \phi(y)
\ni u$ with $y(0)=y_0$. As we have that $E_\epsi(u,y) =
F(u,y)<\infty$, the functional $E_\epsi$ is clearly proper.

Functional $F$ is $\tau$-lower semicontinuous by
  assumption. 
In order to check the $\tau$-lower semicontinuity of $G_{\rm DG}$
  let $(u_\epsi,y_\epsi)\stackrel{\tau}{\to} (u,y)$ be given. With no loss of generality, one can assume $\sup_\epsi G_{\rm
  DG}(u_\epsi,y_\epsi)\leq c<\infty$. 
In particular, we can assume  that $y_\epsi'$ and 
$\partial \phi(u_\epsi)$ are uniformly bounded in $L^2(0,T;H)$. By means of the
chain rule \eqref{eq:chain} we obtain that for all $t \in [0,T]$
\begin{align}
&\phi(y_\epsi(t)) -\phi(y_0) = \int_0^t (\phi\circ y)'\, \dt =  \int_0^t
(\partial \phi(y),y') \dt \nonumber\\
& \quad\leq \|\partial \phi(y) \|_{L^2(0,T,H)}\|y'
\|_{L^2(0,T,H)}<\infty\label{eq:bound}
\end{align}
independently of $t\in [0,T]$ and $\epsi>0$. This implies that $t\mapsto
\phi(y_\epsi(t))$ is uniformly bounded. As the sublevels of $\phi$ are
compact, this yields that there exists $K \subset \subset
H$ such that $y_\epsi(t)\in K$ for all $t\in [0,T]$ and $\epsi>0$. The
uniform bound on  $y_\epsi'$ gives that $y_\epsi$ are equicontinuous
and the Ascoli-Arzel\`a Theorem implies that, up to not relabeled
subsequences, $y_\epsi \to y$ strongly in $C([0,T];H)$. This entails
that $\partial \phi(y_\epsi) \to \partial \phi(u)$ in
$L^2(0,T,H)$ since $\partial \phi$ is strongly $\times$ weakly closed as
subset of $L^2(0,T;H)\times L^2(0,T;H)$. Moreover, the strong
convergence of $y_\epsi$ in $C([0,T];H)$ implies that 
$y_\epsi(T)\to y(T)$ strongly in $H$, so that $\phi(y(T)) \leq \liminf_{\epsi \to
  0}\phi(y_\epsi(T))$ as $\phi$ is lower semicontinuous. Since
$(u_\epsi,y_\epsi')\to (u,y')$ strongly in $L^1(0,T)$, we can pass to lower limits in all terms in
$G_{\rm DG}(u_\epsi,y_\epsi)$ and thus check that
$G_{\rm DG}(u,y)\leq \liminf_{\epsi \to 0}G_{\rm DG}(u_\epsi,y_\epsi)$.

The $\tau$-equicoercivity of
$E_\epsi$ follows as $U$ is compact in $L^2(0,T;H)$ and $G_{\rm DG}(u,y)$ controls
the $L^2(0,T;H)$ norm of $y'$.
\end{proof}

 Before closing this subsection, let us record that in  the former case of
\eqref{eq:case} the two functionals $G_{\rm BEN}$ and $G_{\rm DG}$
coincide. In particular, Figure
\ref{illustration} illustrates the convergence of the penalization via
$G_{\rm DG}$ as well. By considering in that
same linear ODE example
$\phi(y)=\lambda y^2/2$ with $\lambda>0$ instead of $\phi(y)=y^2/2$
one finds the relation $G_{\rm BEN}(u,y)=\lambda G_{\rm
  DG}(u,y)$, which implies that  the minimizers of
$F+\epsi^{-1}G_{\rm BEN}$ and $F+(\epsi/\lambda)^{-1}G_{\rm DG}$
coincide. Hence, for fixed $\epsi>0$ one has that $G_{\rm BEN}$,
respectively $G_{\rm DG}$, delivers the
best approximation in terms of minimum and minimizer if $\lambda<1$, respectively $\lambda>1$. This
in particular proves that, in general, no functional  a priori
dominates  the other in terms
of accuracy of the approximation for fixed $\epsi$.

\subsection{A numerical simulation}

In order to present a second illustration of the penalization procedure,
let us resort to a nonlinear ODE. We consider the optimal control problem 
\begin{equation}
\min \left\{ \frac12\int_0^1(y(t)-1)^2\dt + \frac12 (u-2)^2 \ : \
  y'(t)+y^3(t)=u \ \text{for} \ t \in [0,1], \ y(0)=1
\right\}.\label{ill2}
\end{equation}
with $u  \in \Rz$. By evaluating $u \mapsto F(u,S(u))$ with Matlab,
where $y=S(u)$ is the unique solution to $y'+ y^3=u$ with $y(0)=1$,
one finds a unique optimal $u \sim 1.016$ and, correspondingly,
$F(u,S(u))\sim 0.4917$.

The De Giorgi penalized problem for
$\epsi>0$ reads  
\begin{align*}
&\min\left(F+ \epsi^{-1}G_{\rm DG} \right) = \min \Bigg\{ \frac12\int_0^1(y(t)-1)^2\dt + \frac12 (u-2)^2\\
&\quad +\frac{1}{\epsi}\left(\int_0^1\left(\frac12(y'(t))^2+\frac12(y^3(t){-}u)^2 -u y'(t)
\right) \, \dt + \frac14 y^4(1)-\frac14\right) \ : \ y(0)=1 \Bigg\}.
\end{align*}
The corresponding Euler-Lagrange equations, complemented by the
initial condition, reads
\begin{align}
  - y''(t) + 3(y^3(t)-u)y^2(t) = 0 \ \ \text{for} \ t \in (0,1), \quad
  y'(1) +y^3(1) 
  = u, \quad y(0)=1.\label{eq:el}
\end{align}
Given $u$, 
by numerically solving the latter boundary-value problem with
Matlab, one finds a critical point $y_{\epsi,u}$ of $E_\epsi$
and evaluates $u \mapsto E_\epsi(u,y_{\epsi,u})$. The results of this
simulation are illustrated in Figure \ref{illustration2}, showing
convergence of  minima and  minimizers as $\epsi \to 0$.
\begin{figure}[h]
  \centering
  \pgfdeclareimage[width=105mm]{DG}{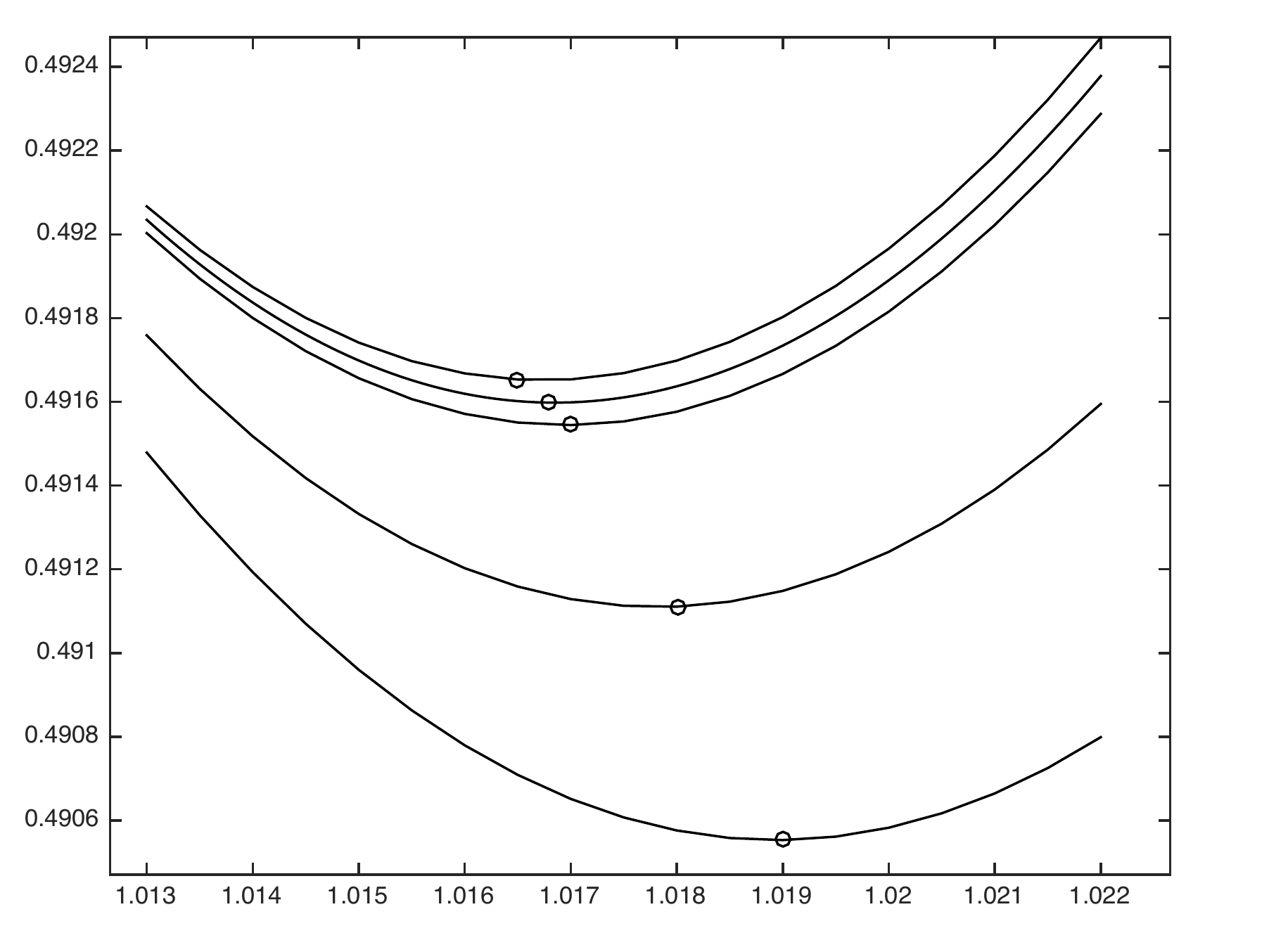}
    \pgfuseimage{DG}
\caption{Curves $u \mapsto E_\epsi(u, y_{\epsi,u})$ for problem  
  \eqref{ill2} for $\epsi=1, \, 0.5,\, 0.1,\, 0.05$, and $0$ (bottom to top). On each curve, the dot indicates the
  minimizer.}
 \label{illustration2}
\end{figure}

\subsection{More general potentials}\label{sec:more} The proof of Theorem \ref{thm:DG}
can be extended to include some more general classes of potentials. 
A first generalization of the theory allows to treat the
case of $\phi=\phi_1 + \phi_2$ with $\partial \phi_1$
not single-valued. In this case, one starts by equivalently rewriting
problem \eqref{eq:gf0} as 
\begin{equation}
  \label{eq:gf0circ}
  y' =\big( u- \partial \phi(y)\big)^\circ \quad \text{in $H$, a.e. in} \ (0,T), \ \ y(0)=y_0.
\end{equation}
Here, $\big( u- \partial \phi(y)\big)^\circ$ denotes the unique element of
minimal norm in the convex and closed set $ u- \partial
\phi(y) = u- \partial \phi_1(y) - D\phi_2(y)$. Let us briefly comment
on the
equivalence of problems \eqref{eq:gf0} and \eqref{eq:gf0circ}. On the
one hand, a solution to \eqref{eq:gf0circ} clearly solves
\eqref{eq:gf0} as well. On the other hand, solutions to \eqref{eq:gf0}
are unique: Let $y_1$ and $y_2$ be two solutions, and write 
$$y_1'-y_2' +\xi_1 - \xi_2 = D\phi_2(y_1) - D\phi_2(y_2)  \quad
\text{in $H$, a.e. in} \ \ (0,T),$$
where $\xi_i \in \partial \phi_1(y_i)$ a.e. in $(0,T)$, for
$i=1,2$. Test the latter equality by $y_1 - y_2$ and integrate on
$(0,t)$. By the monotonicity of $\partial \phi_1$
and the Lipschitz continuity of $D\phi_2$ we obtain 
$$\frac12\| y_1(t) - y_2(t)\|^2 \leq \| D^2 \phi\|_{L^\infty} \int_0^t \|
y_1(s) - y_2(s)\|^2\d s$$
and $y_1=y_2$ follows by the Gronwall Lemma.

Equality \eqref{eq:gf0circ} can then be equivalently recast as
$G_{\rm DG}(u,y)=0$ along with the choice
$$ G_{\rm DG}(u,y)=
\left\{
  \begin{array}{ll}
    &\disp\int_0^T \left(\disp\frac12 \| y'\|^2+ \frac12 \|
(\partial\phi(y)-u)^\circ\|^2 - (u,y') \right)\dt +
\phi(y(T))- \phi(y_0) \\
&{}\qquad  \text{if} \ y \in D(\partial\phi_1) \ \text{a.e.  and} \ y(0)=y_0\\
&\infty \quad  \text{otherwise}.
  \end{array}
\right.$$
Note that $G_{\rm DG}$ is proper, as it vanishes on solutions of the
gradient flow. In particular, if  $G_{\rm DG}(u,y)<\infty$
we have $y\in H^1(0,T;H)$ and we can find $\xi \in L^2(0,T;H)$
such that $\xi -u = (\partial\phi(y)-u)^\circ$ and $\xi \in \partial \phi(y)$ a.e. Then,  by means of the chain rule \eqref{eq:chain} one computes
$$(\phi\circ y)' = (\xi,y') = (\xi -u,y') + (u,y') \quad
\text{a.e. in} \ \ (0,T).$$
as well as the chain of equivalences
\begin{align*}
& y'  =  (u - \partial \phi(y))^\circ \ \ \text{a.e.} \\
&\quad \Leftrightarrow   0 = \frac12 \|y' + \xi-u\|^2 =
  \frac12\|y'\|^2+\frac12 \| \xi-u\|^2 + (\xi-u,y') \ \ \text{a.e.} \\
&\quad \Leftrightarrow   0 =
  \frac12\|y'\|^2+\frac12 \| \xi-u\|^2-(u,y') + (\phi\circ y)' \ \  \text{a.e.} \\
& \quad \Leftrightarrow G_{\rm DG}(u,y)=0.
\end{align*}
In order to extend the results of Theorem \ref{thm:DG} to this case,
one just needs to check that, by replacing the term $ \partial
\phi(y)-u$ with $( \partial
\phi(y)-u)^\circ$ in the functional, coercivity and lower semicontinuity
still hold. As for the first, one still has that $\phi$ is controlled along trajectories
as in \eqref{eq:bound}, since $( \partial
\phi(y)-u)^\circ= \xi -u $ a.e., for some $\xi \in \partial \phi(u)$
a.e. As for lower semicontinuity, one  just needs to be able to pass
to the $\liminf$ in the term containing $(\partial
\phi(y)-u)^\circ$. By letting $y_\epsi \to y$ strongly in $C([0,T];H)$
and $\eta_\epsi  =(u - \partial
\phi(y_\epsi))^\circ \to \eta$ weakly in $L^2(0,T;H)$ one finds that 
$ \xi_\epsi:=u - \eta_\epsi\in \partial \phi(y_n)$ a.e. are such that
$\xi_\epsi  
\to u - \eta =: \xi$ weakly in $L^2(0,T;H)$. Moreover, by the
strong $\times$ weak closure of $\partial \phi$ we have that $\xi
\in \partial \phi(y)$ a.e. We conclude that
\begin{align*}
&\frac12\int_0^T \| (\partial
\phi(y)-u)^\circ\|^2 \dt \leq  \frac12\int_0^T \| \xi -u\|^2 \dt 
  = \frac12 \int_0^T \| \eta\|^2 \dt  \\
& \quad\leq
\liminf_{\epsi \to 0}\frac12\int_0^T \|  \eta_\epsi  \|^2 \dt =\liminf_{\epsi \to 0}\frac12\int_0^T \| ( \partial
\phi(y_\epsi)-u)^\circ\|^2 \dt
\end{align*}
and lower semicontinuity of $G_{\rm DG}$ follows.

Even more generally, the theory  could  be adapted to potential which are
not $C^{1,1}$ perturbations of convex functions. 
The reader is referred to {\sc Rossi \& Savar\'e}
\cite{Rossi-Savare06} where a general frame for existence of
solutions to gradient flows on nonconvex functionals is addressed. In
this context, weaker notions of (sub)differential are introduced  and the
validity of a corresponding chain rule as in \eqref{eq:chain} is
discussed. In
particular, examples of operators fulfilling a suitable chain rule are
presented, including classes of dominated concave perturbations of convex
functions.

Let us mention that the validity of a chain rule {\it equality},
albeit of a paramount importance in order to relate the minimization
of $G_{\rm DG}$ to the solution of \eqref{eq:gf0}, is actually not
needed to prove Theorem \ref{thm:DG}. In fact, the chain rule
\eqref{eq:chain} has been used there just to check that the potential
$\phi$ remains uniformly bounded along trajectories. In particular, a
suitable chain-rule {\it inequality} would serve for this purpose as
well.

\subsection{Generalized gradient flows}
The De Giorgi functional approach can be adapted to encompass {\it
  generalized gradient flows}, namely relations of the form
\begin{equation}
  \label{eq:gf33}
  \partial \psi (y,y') + \partial \phi(y)\ni u \quad \text{for a.e.} \
  t\in (0,T), \quad y(0)=y_0.
\end{equation}
Here, $\psi: H \times H \to [0,\infty)$  and $\partial \psi
(y, y') $ denotes partial subdifferentiation with respect to the second
variable only.
More precisely, we assume that the map $v\in H\mapsto \psi(y,v)$
  is convex and lower
  semicontinuous for all $y\in H$, the map $
  (y,v,w)\in H\times H\times H \mapsto \psi(y,v) + \psi^*(y,w)$ is
  weakly lower semicontinuous
and
\begin{align}
\psi (y,v)+\psi^*(y,w)\geq c \|v\|^p+ c\|w\|^{p'} \quad \forall y,\,
v,\, w\in H 
 \label{eq:psi5}
\end{align}
and some $p>1$
where  $p'=p/(p-1)$ and the
Legendre-Fenchel conjugation is taken with respect to the second
variable only. An example for $\psi$ satisfying \eqref{eq:psi5} is
$\psi(y,y) = \beta(y)|y|^p$, where $p>1$ and $\beta$ is sufficiently smooth,
uniformly positive, and bounded. Note that this includes the case of
{\it doubly nonlinear} flows.  As in Theorem \ref{thm:DG}, we assume for simplicity
that $\partial\phi = \partial \phi_1 + D\phi_2$  and is single-valued.  

Solutions to \eqref{eq:gf33} can be characterized via $G_{\rm
  DG}(u,y)=0$ where $G_{\rm DG}: L^p(0,T;H)\times W^{1,q}(0,T;H) \to
[0,\infty]$  is defined as
\begin{equation}
 G_{\rm DG}(u,y)=
\left\{
  \begin{array}{ll}
   & \disp\int_0^T \left(\disp\psi(y,y') +\psi^*(y,
u{-}\partial \phi(y))- (u,y') \right)\dt +
\phi(y(T)) - \phi(y_0) \\
&\qquad  \text{if} \ y \in D(\partial \phi) \ \text{a.e and} \  y(0)=y_0\\[2mm]
&\infty \quad   \text{otherwise}.
  \end{array}
\right.\label{eq:DG}
\end{equation} 
This can be checked by equivalently rewriting
\begin{align*}
  &\partial \psi (y,y') + \partial \phi(y)\ni u \ \ \text{a.e.} \\
& \Leftrightarrow \psi(y,y')+ \psi^*(y,u-\partial \phi(y)) -
(u-\partial \phi(u),y')=0 \ \ \text{a.e.} \\
& \Leftrightarrow \psi(y,y')+ \psi^*(y,u-\partial \phi(y)) -
(u,y') + (\phi\circ y)'=0 \ \ \text{a.e.} \\
&\Leftrightarrow  G_{\rm DG}(u,y)=0.
\end{align*} 
Indeed, the last equivalence follows by integrating the second-last
relation in time, in one direction, and by realizing that the
integrand is always nonnegative, in the other direction.

By replacing $\|\cdot \|^2/2$ by $\psi(y,\cdot)$ under assumption
 \eqref{eq:psi5}, an analogous statement to  Theorem
\ref{thm:DG} holds. More precisely, by assuming $F:L^p(0,T;H)  \times
W^{1,p'}(0,T;H) \to [0,\infty)$ to be lower
semicontinuous in $X=L^p(0,T;H) \times W^{1,p'}(0,T;H)$ with respect to
the strong$\times$ weak topology and $U$ to be compact in $ L^p(0,T;H)$, one can reproduce the former
argument.  Note however that extra conditions have to be imposed
in such a way that pairs with $G_{\rm BEN}(u,y)=0$ exist.

%%%%%%%%%%%%%%%%%%%%%%%%%%%%%%%%%%%%%

\subsection{GENERIC flows} 
The applicability of the penalization technique via the De Giorgi
functional can be extended to classes of so-called GE\-NER\-IC flows 
(General Equations for Non-Equilibrium  Reversible-Irreversible
Coupling). These are systems of the form
\begin{equation}
  \label{eq:gf4}
  y' = L(y) \,D E(y) -K(y) (\partial \phi(y)-u)  \quad \text{for a.e.} \
  t\in (0,T), \quad y(0)=y_0.
\end{equation}
Here, $-\phi$ is to be interpreted as the {\it entropy} and will have
the property of being
nondecreasing in time. The functional $E: H \to \Rz$ represents an
{\it energy}, to be
conserved along trajectories instead. For the sake of simplicity, we assume
$E$ to be Fr\'echet differentiable, with a linearly bounded, strongly
$\times$ weakly closed
differential $DE$. The mapping $K:H \to {\mathcal L}(H)$ (linear and
continuous operators) is the so called {\it
  Onsager} operator and is asked to be continuous with symmetric and positive
semidefinite values. On the other hand, the operator $L: H \to
{\mathcal L}(H)$ is required to be continuous with
antiselfadjoint values, namely $L^*(y) = -L(y)$.

The GENERIC formalism \cite{Grmela} is a general approach to the variational formulation of physical
models and is particularly tailored to the unified
treatment of coupled conservative and dissipative dynamics. 
Potentials and
operators are related by the following structural assumptions 
\begin{equation}
  \label{eq:gen}
  L^*(y)\partial \phi (y) = K^*(y) D E(y)=0.
\end{equation}
These guarantee that solutions of \eqref{eq:gf4} are such that $(E
\circ y)'=0$ and $(-\phi\circ y)'\geq 0$, namely energy is conserved
and entropy increases along trajectories.
To date, GENERIC has been successfully applied to a variety of situations ranging from
complex fluids \cite{Grmela}, to dissipative quantum mechanics \cite{Mielke13}, to
thermomechanics \cite{generic_souza,Mielke11},
and to the Vlasov-Fokker-Planck equation~\cite{Peletier}. 

By defining the convex potential $\xi \mapsto \psi^*(y,\xi)=(K(y)\xi,\xi)/2$, so that
$K(y)=\partial \psi^*(y,\cdot)$ (subdifferential with respect to the
second variable only), problem
\eqref{eq:gf4} can be reformulated as $G_{\rm DG}(u,y)=0$ where now
$G_{\rm DG}: L^2(0,T;H) \times H^1(0,T;H) \to
[0,\infty]$ is defined as
\begin{equation}
 G_{\rm DG}(u,y)=
\left\{
  \begin{array}{ll}
    &\disp\int_0^T \Big(\disp\psi(y,y'{-}L(y)\,DE(y))  +\psi^*(y,
u{-}\partial \phi(y))\Big)\dt\\
&\quad {-} \disp\int_0^T  (u,y'{-}L(y)\, DE(y))  \, \d t   +
\phi(y(T)) - \phi(y_0) \\[4mm]
&\qquad \text{if} \ y \in D(\partial \phi) \ \text{a.e. and} \  y(0)=y_0\\[3mm]
&\infty \quad   \text{otherwise}.
  \end{array}
\right.\label{eq:DG}
\end{equation} 
In fact, we have the following chain of equivalencies
\begin{align*}
&y' = L(y)\,D E(y) -K(y) (\partial \phi(y)-u) \ \ \text{a.e.}\\
 &\Leftrightarrow\ 
  \psi(y,y'{-}L(y)\,DE(y))+\psi^*(y,u{-}\partial\phi(y)) 
	- \big( y'{-}L(y)\,DE(y), u{-} \partial \phi(y)\big) =0 \ \ \text{a.e.}\\
& \Leftrightarrow \ \psi(y,y'{-}L(y)\,D
  E(y))+\psi^*(y,u{-}\partial\phi(y)) \\
&\qquad -  (u,y'-L(y)\, DE(y))
  -(DE(y),L^*(y)\partial \phi(y)) + ( \phi\circ y)' =0 \ \  \text{a.e.}  \\
&\Leftrightarrow\
  G_{\rm DG}(u,y)=0.
\end{align*}
Again, the last equivalence follows by integration in time.
 
The statement of Theorem \ref{thm:DG} can be extended to cover the
case of GENERIC flows as well.  Let us assume from the very
beginning that  for all $u\in
H$ there exists $y$ such that $G_{\rm DG}(u,y)=0$. In applications $K$
and $\phi$ are often
degenerate (see below). Coercivity for the sole $G_{\rm DG}$ is hence
not to be expected. In order to state a general result, let us
hence assume $F$ itself to be lower semicontinuous and coercive with
respect to the strong $\times$ weak topology of $L^2(0,T;H) \times
H^1(0,T;H)$. Moreover, let $F$ be coercive with
respect to the strong $\times$ strong topology of $L^2(0,T;H)\times C([0,T];H)$ on
sublevels of $\phi$ and to
control the $L^2(0,T;H) $ norm of $\partial \phi(y)$ (alternatively,
 let 
$\partial \phi(y)$ be linearly bounded). Eventually, we ask
$\psi^*$ and
$\psi$ to be lower semicontinuous 
in the following sense
\begin{align}
&\psi(y,\eta)+\psi^*(y,\xi) \leq \liminf_{\epsi \to
  0}\big(\psi(y_\epsi,\eta_\epsi)+\psi^*(y_\epsi,\xi_\epsi)  \big)\nonumber\\
&\quad\forall y_\epsi \to y \ \ \text{strongly in} \ \ C([0,T];H) \ \
  \text{with} \ \ \sup \phi(y_\epsi(t)) <\infty\nonumber\\
&\quad \text{and}  \ \ (\eta_\epsi,\xi_\epsi) \to (\eta,\xi) \ \
\text{weakly in} \ \ L^2(0,T;H)^2.\label{eq:mosco}
\end{align}

Owing to the assumptions on $F$, in order to reproduce the argument of Theorem \ref{thm:DG} in this
setting, one is left to check the lower semicontinuity of $G_{\rm
  DG}$. Let
$(u_\epsi,y_\epsi)\to (u,y)$ strongly $\times$ weakly in $L^2(0,T;H)
\times H^1(0,T;H)$ and assume with no loss of generality that
$\partial \phi(y_\epsi)$ is bounded in $L^2(0,T;H)$. By arguing as in
\eqref{eq:bound} one can bound $t\mapsto \phi(y_\epsi(t))$ so that all
trajectories belong to a sublevel of $\phi$. From the strong coercivity
of $F$ on sublevels of $\phi$ we deduce strong compactness in
$C([0,T];H)$  for $y_\epsi$,  so that
$y_\epsi\to y$ uniformly, up to not relabeled subsequences. As $DE$ is
assumed to be strongly $\times$ weakly closed and $L$ is continuous, we have that $y_\epsi'-L(y_\epsi)\, DE(y_\epsi) \to
y'-L(y)\, DE(y)$ weakly in $L^2(0,T;H)$. On the other hand, the
strong $\times$ weak closure of $\partial \phi$ ensures that, again
without relabeling, $\partial \phi(y_\epsi) \to \partial \phi(y)$
weakly in $L^2(0,T;H)$. We can hence make use of \eqref{eq:mosco} and
deduce the lower semicontinuity of
$G_{\rm DG}$.

Before closing this discussion, let us give an example of an elementary GENERIC system
 fitting into  this abstract setting. Consider the thermalized oscillator problem
\begin{align}
 &  q'' + \nu q' + \lambda q + \theta =0,\label{eq:o1}\\
&\kappa  \theta' = \nu ( q')^2 + \theta q'. \label{eq:o2}
\end{align}
Here, $y=(q,p,\theta)\in \Rz^3=:H$ where $q$ represents the state of the
oscillator, $p$ is its momentum, and $\theta>0$ is the absolute
temperature. The nonnegative constants $\nu$, $\lambda$, and $\kappa$
are the viscosity parameter, the elastic  modulus,  and the heat capacity,
respectively. Relations \eqref{eq:o1} and \eqref{eq:o1} express the
conservation of momentum and energy, respectively.

In order to reformulate \eqref{eq:o1}-\eqref{eq:o2} as a GENERIC
system, we specify the free energy of the system as
$$\Psi(y)= \frac{\lambda}{2}q^2 + q\theta - \kappa\theta \ln
\theta.$$
Moving  from  this, the entropy $-\phi$ and the total energy $E$ are
derived by the classical Helmholtz relations as 
\begin{align*}
-\phi(y)= -\partial_\theta \Psi= -q+\kappa\ln \theta +\kappa, \quad 
E(y)=   \frac{1}{2}p^2+ \Psi+\theta \phi = \frac{1}{2}p^2 +\frac{\lambda}{2}
  q^2+ \kappa\theta.
\end{align*}
In particular, we have that 
$$DE(y) = (\lambda q, p, \kappa), \quad \partial \phi(y) = (-1,0,\kappa/\theta).$$
By defining the mappings $K$ and $L$ as
$$ K(y)=\nu \theta 
\left(
 \disp \begin{matrix}
    0 & 0 & 0 \\
0 & 1 & -p/\kappa\\
0 & -p/\kappa & p^2/\kappa^2
  \end{matrix}
\right), \quad 
L(y)=
\left(
 \disp \begin{matrix}
    0 & 1& 0 \\
-1 &0 &\disp -\theta/\kappa\\
0 &\disp \theta/\kappa& 0
  \end{matrix}
\right),
$$
we readily check that the compatibility conditions \eqref{eq:gen} hold
and that system \eqref{eq:o1}-\eqref{eq:o2} takes the form in
\eqref{eq:gf4}. By computing the conjugate we find
$$\psi^*(y,\xi) = \frac{\nu \theta}{2}(\xi_2 - p \xi_3/\kappa)^2, \quad
\psi(y,\eta) = 
\left\{
  \begin{array}{ll}
    \disp\frac{1}{2\nu \theta}\eta_2^2& \quad \text{if} \
                                   \eta_1=\eta_3+py_2/\kappa=0,\\
\infty&\quad \text{otherwise}
  \end{array}
\right.
$$
for all $y=(q,p,\theta)\in \Rz^3$ with $\theta>0$ and for all
$(\xi,\eta)\in \Rz^2$.
In particular, the lower semicontinuity \eqref{eq:mosco} follows as
$\sup \phi(y_\epsi(t))<\infty$ implies that $\theta_\epsi\geq
c>0$ for some $c$, hence $1/\theta_\epsi \to 1/\theta$ in $C([0,T])$.

In order to give a concrete example of target functional $F$ choose
\begin{align*}
F(u,y) &= 
\frac12\int_0^T |y-y_{\rm target}|^2\d t + \frac12\int_0^T |y'-y_{\rm
    target}'|^2\d t +\int_0^T |1/\theta - 1/\theta_{\rm target}|^2\d
  t\\
&+ \int_0^T |u|^2 \, \dt +  \int_0^T |u'|^2 \, \dt
\end{align*}
for some given  $y_{\rm target} = (q_{\rm target},p_{\rm target},\theta_{\rm
  target}) \in H^1(0,T;H)$ with $1/\theta_{\rm target} \in
L^2(0,T)$. The functional $F$ is coercive with
respect to the strong $\times$ weak topology of $L^2(0,T;H) \times
H^1(0,T;H)$, as well as to the strong$\times$strong topology of $L^2(0,T;H)\times C([0,T];H)$ on
sublevels of $\phi$. Moreover, it controls the $L^2(0,T;H) $ norm of
$\partial \phi(y)$. Hence, the abstract setting described above
applies.

%%%%%%%%%%%%%%%%%%%%%%%%%%%%%%%%%%%%%%%%%%%%%%

\section*{Acknowledgement} 
This work has been funded by the Vienna Science and Technology Fund (WWTF)
through Project MA14-009 and  by the Austrian Science Fund (FWF)
projects F\,65 and I\,2375.

\bibliographystyle{alpha}

\begin{thebibliography}{99}



\bibitem{Auchmuty93}
{G.~Auchmuty}. {Saddle-points and existence-uniqueness for evolution
  equations}. {\it Differential Integral Equations}, 6 (1993), 1161--1171.


\bibitem{generic_souza}
F. Auricchio, E. Boatti, A. Reali, U. Stefanelli.
 {Gradient structures for the thermomechanics of shape-memory
   materials}. {\it  Comput. Methods Appl. Mech. Engrg.} 299 (2016),
 440--469.

\bibitem{Bergounioux92}
M. Bergounioux. A penalization method for optimal control of elliptic
problems with state constraints. {\it SIAM J. Control Optim.} 30
(1992), 305--323. 


\bibitem{Bergounioux94}
M. Bergounioux. Optimal control of parabolic problems with state
constraints: a penalization method for optimality conditions. {\it
  Appl. Math. Optim.} 29 (1994),  285--307.

\bibitem{Bergounioux98}
M. Bergounioux. Optimal control of problems governed by abstract
elliptic variational inequalities with state constraints. {\it SIAM
  J. Control Optim.} 36 (1998),  273--289. 



\bibitem{Brezis73} 
{H.~Br\'ezis}.
\newblock {\it Operateurs maximaux monotones et semi-groupes de 
         contractions dans les espaces de Hilbert}.
\newblock {Math Studies, Vol. 5, North-Holland, Amsterdam/New York}
(1973).


\bibitem{Brezis-Ekeland76}
 {H.~Br\'ezis and I.~Ekeland}. {Un principe
  variationnel associ\'e \`a certaines \'equations paraboliques. {L}e cas
  ind\'ependant du temps}. {\it C. R. Acad. Sci. Paris S\'er. A-B},
282 (1976),  A971--A974.

\bibitem{Brezis-Ekeland76b}
{ H.~Br\'ezis and I.~Ekeland}. {Un principe variationnel associ\'e \`a
  certaines \'equations paraboliques. {L}e cas d\'ependant du temps}.
{\it C. R.
  Acad. Sci. Paris S\'er. A-B}, 282 (1976),  A1197--A1198.

\bibitem{browder}
F. Browder and P. Hess. Nonlinear mappings of monotone type in Banach
spaces. {\it J. Funct. Anal.} 11 (1972), 251--294.

\bibitem{DalMaso93}
G.~{Dal Maso}.
\newblock {\em An introduction to {$\Gamma$}-convergence}.
\newblock Progress in Nonlinear Differential Equations and their Applications,
  8. Birkh\"auser Boston Inc., Boston, MA, 1993.

\bibitem{DeGiorgi79}
{E.~De Giorgi and T.~Franzoni}.
On a type of variational convergence. In {\em Proceedings of the
  Brescia Mathematical Seminar}, 9, 63--101, Milan, 1979. 
\bibitem{DeGiorgi80}
E.~De Giorgi, A.~Marino, and M.~Tosques.
\newblock Problems of evolution in metric spaces and maximal decreasing curve.
\newblock {\em Atti Accad. Naz. Lincei Rend. Cl. Sci. Fis. Mat. Natur. (8)}.
  68 (1980),  180--187.


\bibitem{Peletier}
M.~H.~Duong, M.~A.~Peletier, J.~Zimmer. GENERIC formalism of a
Vlasov--Fokker--Planck equation and connection to large-deviation
principles. {\it Nonlinearity}, 26 (2013), 2951--2971.

\bibitem{Fitzpatrick}
S. Fitzpatrick. Representing monotone operators by convex
functions. Workshop/Miniconference on Functional Analysis and
Optimization (Canberra, 1988), 59–65, {\it Proc. Centre Math.
Anal. Austral. Nat. Univ.} 20, Austral. Nat. Univ., Canberra, 1988.

\bibitem{Gariboldi09}
C. M. Gariboldi and D. A. Tarzia. Convergence of distributed optimal
controls in mixed elliptic problems by the penalization method. {\it Math. Notae}, 45 (2007/08), 1–19. 


\bibitem{Ghoussoub-McCann04}
{N.~Ghoussoub and R.~J. McCann}. {A least action principle for steepest
  descent in a non-convex landscape}. In Partial differential equations and
  inverse problems, vol.~362 of {\it Contemp. Math.}, Amer. Math. Soc., Providence,
  RI, 2004, pp.~177--187.

\bibitem{Ghoussoub-Tzou04}
N.~Ghoussoub, L.~Tzou.
\newblock A variational principle for gradient flows.
\newblock {\em Math. Ann.}, 330 (2004), 519--549.


\bibitem{Ghoussoub08}
N.~Ghoussoub.
\newblock {\em Selfdual partial differential systems and their variational
  principles}.
\newblock Universitext. Springer, New-York, 2009.


\bibitem{Grmela}
M.~Grmela and H.~C.~\"Ottinger. {Dynamics and thermodynamics of
  complex fluids. I. Development of a general formalism}. {\em
  Phys. Rev. E}, 56 (1997), 6620--6632.

\bibitem{Gunzburger00}
M. D. Gunzburger and H.-C. Lee. A penalty/least-squares method for
optimal control problems for first-order elliptic systems. {\it
  Appl. Math. Comput.} 107 (2000),  57--75.

\bibitem{Kenmochi77}
N. Kenmochi. Nonlinear evolution equations with variable domains in
Hilbert spaces. {\it Proc. Japan Acad. Ser. A Math. Sci.} 53 (1977),
163--166.

\bibitem{Kenmochi91}
N. Kenmochi,  T. Koyama. Nonlinear functional variational inequalities
governed by time-dependent subdifferentials. {\it Nonlinear Anal.} 17
(1991), 863--883.

\bibitem{Lemaire96}
{B.~Lemaire}. {An asymptotical variational principle associated with
  the steepest descent method for a convex function}. {\it J. Convex Anal.} 3
  (1996), 63--70.

\bibitem{Lions68}
J.-L. Lions. {\it Contr\^ole optimal de syst\`emes gouvern\'es par des
  \'equations aux d\'eriv\'ees partielles}. Gauthier-Villars, Paris 1968.


\bibitem{Mabrouk00}
M.~Mabrouk. A variational
  approach for a semi-linear parabolic equation with measure data. {\it Ann. Fac.
  Sci. Toulouse Math.} 9 (2000), 91--112.


\bibitem{Mabrouk03}
M.~Mabrouk. A variational
  principle for a nonlinear differential equation of second order. {\it Adv. in
  Appl. Math.} 31 (2003), 388--419.


\bibitem{Mielke11}
A.~Mielke. {Formulation of thermoelastic dissipative material behavior
using GENERIC}. {\it Contin. Mech. Thermodyn.} 23 (2011), 233--256.

 \bibitem{Mielke13}
 A.~Mielke. Dissipative quantum mechanics using GENERIC. Proc. of the
 conference on {\it Recent Trends in Dynamical Systems}, vol. 35 of
 Proceedings in Mathematics \& Statistics, Springer, 2013,
 pp. 555--585.

\bibitem{Mophou11}
G. Mophou and G.~M. N'Gu\'er\'ekata. Optimal control of a fractional
diffusion equation with state constraints. {\it Comput. Math. Appl.}
62 (2011), 1413--1426.

\bibitem{Moreau77}
J.-J. Moreau. Evolution problem associated with a moving convex set in
a Hilbert space. {\it J. Differential Equations}, 26 (1977),  347-374.

\bibitem{Nayroles76}
{B.~Nayroles}. { Deux th\'eor\`emes de minimum pour certains syst\`emes
  dissipatifs}. {\it C. R. Acad. Sci. Paris S\'er. A-B}, 282 (1976),
  A1035--A1038.

\bibitem{Nayroles76b}
{B.~Nayroles}. { Un th\'eor\`eme de
  minimum pour certains syst\`emes dissipatifs. {V}ariante hilbertienne}.
  {\it Travaux S\'em. Anal. Convexe}, 6 (1976), 22.

\bibitem{Rios76}
{H.~Rios}. {\'{E}tude de la question d'existence pour certains
  probl\`emes d'\'evolution par minimisation d'une fonctionnelle
  convexe}. {\it C.
  R. Acad. Sci. Paris S\'er. A-B}, 283 (1976), A83--A86.

\bibitem{Rossi-Savare06}
R.~Rossi and G.~Savar\'e.
\newblock {Gradient flows of non convex
functionals in Hilbert spaces and applications}.
\newblock {\em ESAIM Control Optim. Calc. Var.} 12 (2006), 564--614.


\bibitem{Roubicek00}
{T.~Roub{\'{\i}}{\v{c}}ek}. {Direct method for parabolic problems}. {\it
  Adv. Math. Sci. Appl.} 10 (2000), 57--65.

\bibitem{be}
U. Stefanelli.
The Brezis-Ekeland principle for doubly nonlinear equations.
{\it SIAM J. Control Optim.} 47 (2008), 1615--1642 

\bibitem{be2}
{U.~Stefanelli}.
\newblock {The discrete {B}rezis-{E}keland principle}.
\newblock {\it J. Convex Anal.}  16 (2009), 71--87.

\bibitem{plas}
{U.~Stefanelli}.
\newblock {A variational principle for hardening elasto-plasticity}.
\newblock {\it SIAM J. Math. Anal.} 40 (2008),  623--652.

\bibitem{visintin08}
A. Visintin. Extension of the Brezis-Ekeland-Nayroles principle
to monotone operators. {\it Adv. Math. Sci. Appl.} 18 (2008),
633--650. 

\bibitem{visintin17}
A. Visintin. On the variational representation of monotone
operators. {\it Discrete Contin. Dyn. Syst. Ser. S},  10 (2017),
909-–918. 

\bibitem{visintin18}
A. Visintin. Structural compactness and stability of semi-monotone
flows. {\it SIAM J. Math. Anal.} 50 (2018), 2628--2663.

\bibitem{Yamada76}
Y. Yamada. On evolution equations generated by subdifferential
operators. {\it J. Fac. Sci. Univ. Tokyo Sect. IA Math.} 23 (1976), 491--515. 

\end{thebibliography}

\end{document}